\documentclass[12pt]{amsart}
\usepackage{amscd,amssymb,amsmath,latexsym,graphicx,amsxtra}
\usepackage[all]{xy}
\newtheorem{theorem}{Theorem}[section]

\def\Z{{\mathbb Z}}
\theoremstyle{definition}

\newtheorem{example}[theorem]{Example}
\newtheorem{proposition}[theorem]{Proposition}
\newtheorem{corollary}[theorem]{Corollary}
\newtheorem{remark}[theorem]{Remark}
\numberwithin{equation}{section}
\newcommand{\pN}{N\times N,N\times N \setminus \Delta(N)}
\begin{document}
\title{Obstruction theory for coincidences of multiple maps}
\author{Tha\'is Monis}
\address{Dept. de Matem\'atica - UNESP Rio Claro, Av. 24A, 1515 - Bela Vista, Rio Claro-SP, Brasil}
\email{tfmonis@rc.unesp.br}
\author{Peter Wong}
\address{Department of Mathematics, Bates College, Lewiston, ME 04240, U.S.A.}
\email{pwong@bates.edu}
\thanks{This work was initiated during the second author's visit to the Department of Mathematics at UNESP-Rio Claro, May 18 - June 17, 2014 and completed during his visit to the department Feb. 21 - 26, 2016. The second author would like to thank the Mathematics Department for the invitation and financial support. The first author was supported by FAPESP of Brazil Grant number 2014/17609-0.}

\begin{abstract} Let $f_1,..., f_k:X\to N$ be maps from a complex $X$ to a compact manifold $N$, $k\ge 2$. In previous works \cite{BLM,MS}, a Lefschetz type theorem was established so that the non-vanishing of a Lefschetz type coincidence class $L(f_1,...,f_k)$ implies the existence of a coincidence $x\in X$ such that $f_1(x)=...=f_k(x)$. In this paper, we investigate the converse of the Lefschetz coincidence theorem for multiple maps. In particular, we study the obstruction to deforming the maps $f_1,...,f_k$ to be coincidence free. We construct an example of two maps $f_1,f_2:M\to T$ from a sympletic $4$-manifold $M$ to the $2$-torus $T$ such that $f_1$ and $f_2$ cannot be homotopic to coincidence free maps but for {\it any} $f:M\to T$, the maps $f_1,f_2,f$ are deformable to be coincidence free.
\end{abstract}
\date {\today}
\maketitle

\section{Introduction}
The celebrated Lefschetz coincidence theorem states that for any two maps $f,g:M\to N$ between closed connected oriented triangulated $n$-manifolds, if the Lefschetz coincidence number (trace) $L(f,g)$ is non-zero then the coincidence set $C(f,g)=\{x\in M \mid f(x)=g(x)\}$ must be non-empty. However, the converse does not hold in general. In this direction, E. Fadell showed \cite{fadell} that if $N$ is simply-connected then the vanishing of $L(f,g)$ is sufficient to deform the maps $f\sim f', g\sim g'$ so that $C(f',g')=\emptyset$. For non-simply connected $N$, the vanishing of the Nielsen number $N(f,g)$ often provides the converse. Following \cite{fadell} and \cite{fadell-husseini}, it was shown in \cite{GJW} that the (primary) obstruction $o_n(f,g)$ ($n\ge 3$) to deforming $f$ and $g$ to be coincidence free is Poincar\'e dual to the twisted Thom class of the normal bundle of the diagonal $\Delta(N)$ in $N\times N$ with appropriate local coefficients.

Suppose $M$ is a compact topological space, $N$ is a closed connected oriented manifold, and $f_1,...,f_k: M\to N$ are maps. A Lefschetz type coincidence class $L(f_1,...,f_k)$ was defined in \cite{BLM} and it was shown that $L(f_1,...,f_k)\ne 0$ implies $C(f_1,...,f_k):=\{x\in M \mid f_1(x)=...=f_k(x)\}$ is non-empty. This Lefschetz type result has been extended to non-orientable $N$ in \cite{MS}.

The purpose of this paper is to examine the converse of this Lefschetz type theorem, that is, the problem when
$$
L(f_1,...,f_k)=0 \Rightarrow  f_1\sim f_1',...,f_k\sim f_k' \text{~such that~} C(f_1',...,f_k')=\emptyset.
$$ 
The approach here is via obstruction theory, following \cite{fadell, fadell-husseini, dobrenko,GJW}. We examine the primary obstruction to deforming $f_1,...,f_k$ to be coincidence free on the $(k-1)n$-th skeleton where $n=\dim N$. We prove an analogous converse of the Lefschetz type coincidence theorem when $N$ is simply-connected and $\dim M=(k-1)n$. We give further examples of Jiang-type spaces $N$ for which the converse theorem holds. This paper is organized as follows. In section 2, we generalize the Lefschetz coincidence classes defined in \cite{BLM} and in \cite{MS} to homomorphisms $\Lambda (f_1,...,f_k;R_{N^k})$ with arbitrary local coefficients. This homomorphism is similar to a certain homomorphism $\phi$ of \cite{daci-peter}, relating to the preimage of a map. In section 3, we study the primary obstruction to deforming $f_1,...,f_k$ to be coincidence free using the appropriate  local coefficient system $\pi_j^*(F)$. The calculation of the local system $\pi_j^*(F)$ is carried out in section 4. We then compute the obstruction to deformation in section 5 and prove the converse of the Lefschetz coincidence theorem for multiple maps in section 6. We also compare our result to similar Nielsen type results of P. Staecker \cite{St}. In section 7, we illustrate our results with examples. In particular, we construct two maps $f_1,f_2: M\to N$, $\dim M>\dim N$ such that $f_1,f_2$ are not deformable to be coincidence free but for {\it any} $f:M\to N$, $f_1,f_2,f$ are homotopic to be coincidence free (see Examples \ref{2-sphere} and \ref{torus}).

\section{Lefschetz coincidence homomorphism with local coefficients}

In this section, we review the Lefschetz coincidence classes introduced in \cite{BLM} and in \cite{MS} and compare them with a more general homomorphism studied in \cite{daci-peter}.

Let $X$ be a compact topological space and $N$ be a connected closed $n$-manifold. For now, let us assume $N$ is oriented. Suppose $f_1,...,f_k:X\to N$ are maps, $\mu \in H^n(\pN)$ be the Thom class of the normal bundle of the diagonal $\Delta(N)$ in $N\times N$. Then, in \cite{BLM}, the authors defined
\begin{equation}\label{BLM-L}
L_1(f_1,...,f_k)=[(j\times ...\times j)\circ (h_1,...,h_{k-1})]^*(\mu \times ...\times \mu)
\end{equation}
where $h_i:X\to N\times N$ is given by $h_i(x)=(f_i(x),f_{i+1}(x))$ and $j:N\times N \hookrightarrow (\pN)$ is the inclusion. In \cite{MS}, a similar Lefschetz type coincidence class was defined as follows. Let $\Delta_k(N)=\{(x,...,x)\in N^k \mid x\in N\}$ be the diagonal of $N$ in $N^k$ and $\mu_k$ be the Thom class of the normal bundle of $\Delta_k(N)$ in $N^k$ in $H^{(k-1)n}(N^k,N^k\setminus \Delta_k(N); R\times \Gamma^*_N \times ...\times \Gamma^*_N)$. Here $R$ is a principal ideal domain and $\Gamma^*_N={\rm Hom}(\Gamma_N,R)$ for the $R$-orientation system $\Gamma_N$ on $N$. Define $\tilde h_i:X\to N$ by $\tilde h_i(x)=(f_1(x),f_{i+1}(x))$. Then 
\begin{equation}\label{MS-L}
\begin{aligned}
L_2(f_1,...,f_k)&=[i \circ (f_1,...,f_{k})]^*(\mu_k) \\
                       &= [(j\times ...\times j)\circ (\tilde h_1,...,\tilde h_{k-1})]^*(\mu \times ...\times \mu)
\end{aligned}
\end{equation}
where $i:N^k \to (N^k, N^k\setminus \Delta_k(N))$ is the inclusion. Note that 
$$(j\times ...\times j)\circ (h_1,...,h_{k-1})=\sigma \circ i \circ (f_1,...,f_k)$$
where $\sigma (x_1,...,x_k)=((x_1,x_2), (x_2,x_3),...,(x_{k-1},x_k))$. The induced homomorphism $\sigma^*:\otimes_{k-1} [H^n(N^2,N^2\setminus \Delta(N))] \to H^{(k-1)n}(N^k,N^k\setminus \Delta_k(N)$ is an isomorphism when $N$ is orientable. In this case, $\sigma^*(\mu \times ...\times \mu)=\mu_k$ and hence $L_1(f_1,...,f_k)=L_2(f_1,...,f_k)$.

A close inspection indicates that $L_2(f_1,...,f_k)$ is in fact a special case of the coincidence index homomorphism studied in \cite{daci-peter} which we recall as follows. Let $X,Y$ be closed manifolds of dimension $a$ and $b$, respectively. Suppose $B\subset Y$ is a closed submanifold of dimension $\ell$. Let $F:X\to Y$ be transverse to $B$ so that $C=F^{-1}(B)$ is a closed submanifold of dimension $a-b+\ell$. Suppose $R_Y$ is a local system on $Y$ with typical group $R$, $R_X=F^*(R_Y), R'_X=R_X \times \Gamma_X$ and $R'_C=i_1^*(R'_X)$ where $i_1:C=F^{-1}(B) \hookrightarrow X$. For any $0\le j$, define $\phi(j)$ to be the composite homomorphism given by
$$
H^j(Y,Y\setminus B;R_Y) \stackrel{F^*}{\longrightarrow} H^j(X,X\setminus C;R_X) \stackrel{A^{-1}}{\longrightarrow} H_{a+b-j}(C;R'_C)
$$
where $A$ represents the Alexander Duality isomorphism (see e.g. \cite{Spanier}). It was shown in \cite[Theorem 2.4]{daci-peter} that for any $r\in H^j(Y,Y\setminus B;R_Y)$, 
$${i_1}_*\phi(j)(r)=F^*(j^*_2(r))\cap [z_X]$$
where $j_2:Y\hookrightarrow (Y,Y\setminus B)$ and $[z_X]$ is the (twisted) fundamental class of $X$ with coefficients in $R_X$. It follows from sections 4.3 and 4.5 of \cite{daci-peter} that $\phi((k-1)n)\ne 0 \Rightarrow C\ne \emptyset$. Now let $a=(k-1)n, b=kn, Y=N^k, B=\Delta_k(N), F=(f_1,...,f_k), C=C(f_1,...,f_k)$ and $j=(k-1)n$. If $R_Y=R\times \Gamma^*_N \times ...\times \Gamma^*_N$ then it is easy to see that
\begin{equation}\label{L2-phi}
L_2(f_1,...,f_k)=I^* \circ A\circ \phi((k-1)n)(\mu_k)
\end{equation}
where $I:X \hookrightarrow (X,X\setminus C)$ is the inclusion. Now $L_2(f_1,...,f_k)\ne 0$ implies that $\phi((k-1)n)\ne 0$. Thus, the Lefschetz type coincidence theorems of \cite{MS} and of \cite{BLM} follow immediately.

Next, we define a Lefschetz coincidence homomorphism with any arbitrary local coefficient system $R_{N^k}$.

Let $X$ be a compact topological space and $N$ a closed manifold of dimension $n$. Suppose $f_1,...,f_k:X\to N$ are maps. Let $F=(f_1,...,f_k):(X,X\setminus C(f_1,...,f_k))\to (N^k,N^k\setminus \Delta_k(N))$, $I:X\hookrightarrow (X,X-C(f_1,...,f_k))$. Define the {\it Lefschetz coincidence homomorphism of $f_1,...,f_k$ with local coefficients $R_{N^k}$} to be

\begin{equation}\label{Lef-R_N}
\begin{aligned}
\Lambda (f_1,...,f_k;R_{N^k}):=I^*\circ F^*: H^{(k-1)n}(N^k,N^k\setminus \Delta_k(N); R_{N^k})\to H^{(k-1)n}(X; F^*(R_{N^k})).
\end{aligned}
\end{equation}

Thus, we have
\begin{equation}\label{Lef-R_N2}
\Lambda (f_1,...,f_k;R\times \Gamma^*_N\times ...\times \Gamma^*_N)(\mu_k):=i^*\circ (f_1,...,f_k)^*(\mu_k)=L_2(f_1,...,f_k)
\end{equation}
where $i: N^k \hookrightarrow (N^k,N^k\setminus \Delta_k(N))$. Now, the following result, which generalizes the main theorems of \cite{BLM,MS}, is immediate.

\begin{theorem}\label{Lef_thm}
Let $X$ be a compact topological space, $N$ a closed manifold of dimension $n$ and $k\ge 3$. For any maps $f_1,...,f_k:X\to N$ and local coefficient $R_{N^k}$, if $\Lambda (f_1,...,f_k; R_{N^k})\ne 0$ then $C(f_1,...,f_k)\ne \emptyset$.
\end{theorem}

In section 5, we will generalize the Lefschetz coincidence classes of \cite{BLM,MS} to $\Lambda (f_1,...,f_k; R_{N^k})(\tilde \mu_k)$ for an approriate local coefficient and cohomology class $\tilde \mu_k$ which will be called the twisted Thom class of $N^k$. 

\section{Obstructions and local coefficients}

In \cite{fadell}, E. Fadell showed that the classical Lefschetz coincidence class $L(f,g)$ coincides with the obstruction to deforming $f$ and $g$ to be coincidence free provided the target manifold is simply-connected. For non-simply connected manifolds, the obstruction class \cite{fadell-husseini} necessarily involves certain local coefficient system. We now recall, following the treatment in \cite{fadell-husseini} or in \cite{dobrenko}, the local coefficients employed in the obstruction to deformation.

Let $f:X\to Y$ be a map from a finite connected complex $X$ to a closed connected manifold $Y$ and $B\subset Y$ be a closed submanifold. First, we replace the inclusion $i: Y\setminus B \hookrightarrow Y$ by a fiber map $p:E\to Y$ where $E=\{(x,\omega)\mid x\in Y\setminus B, \omega \in Y^{[0,1]}, \omega (0)=x\}$ and $p(x,\omega)=\omega(1)$. Let the typical fiber be $F=p^{-1}(x_0)=\{(x,\omega)\mid \omega(1)=x_0\}$ for some $x_0\in Y\setminus B$. It follows that $\pi_{j+1}(F;\overline{x_0})\cong \pi_j(Y,Y\setminus B;x_0)$. 

It is straightforward to verify that $f$ is deformable into $Y\setminus B$, i.e., $f\sim f'$ so that $f'(X)\subseteq Y\setminus B$, if and only if $f$ can be lifted to a map $\tilde f:X\to E$ such that $p\circ \tilde f=f$. Furthermore, $f$ is deformable into $Y\setminus B$ if and only if the pullback fibration $q: f^*(p)\to X$ induced by $f$ has a (global) section. 

If $(Y,Y\setminus B)$ is $(m-1)$ connected. It follows from the classical obstruction theory for lifting (see e.g. \cite{whitehead}) that the primary obstruction to finding a section to $q$ is given by a class $o_m(f)\in H^m(X;\pi^*_{m-1}(F))$ where $\pi^*_{m-1}(F)$ is the local system on $X$ induced by $f:X\to Y$ from the local system $\pi_{m-1}(F)$ on $Y$. This class is the obstruction to deforming $f$ to $f'$ with $f'(X^{(m)})\subset Y\setminus B$ where $X^{(m)}$ is the $m$-skeleton of $X$.

Now let $X=M, Y=N^k, B=\Delta_k(N)$ and $f=(f_1,...,f_k)$  where $f_1,..., f_k:M \to N$ are maps from a finite connected CW-complex $M$ to a closed connected manifold $N$ of dimension $n$, where $k\ge 2$. Next, we will study the obstruction to deforming the maps $f_1,...,f_k$ to be coincidence free but first we must calculate the local coefficient system $\pi^*_{m-1}(F)$.

\
In this setting, we set $$E=E(N^k, N^k \setminus \Delta_k(N)) = \{ (y, \omega) \ | \  y \in N^k \setminus \Delta_k(N) \ \text{and} \  \omega: I \to N^k \ \text{is a path with} \ \omega(0)=y \}.$$
Let $p: E \to N^k$ be the fiber map given by $p(y, \omega)=\omega(1)$. If we take a point $y_0=(y_0^1, \ldots , y_0^k) \in N^k$ the fiber over $y_0$, $F=p^{-1}(y_0)$, is given by $$F=\{(\omega(0), \omega) \  |   \  \omega(1)=y_0 \}.$$ In words, the fiber is the space of paths in $N^k$ which begins in $N^k \setminus \Delta_k(N) $ and end at $y_0$.  Let $y_0 \in N^k \setminus \Delta_k(N)$ and denote by $\overline{y_0}$ the constant path $\overline{y_0}(t)=y_0$ for every $t$. There is an isomorphism $$ \pi_j(F; \overline{y_0}) \simeq \pi_{j+1}(N^k, N^k \setminus \Delta_k(N); y_0).$$ The above identification can be seen as the following:  the element $[\alpha] \in \pi_j(F; \overline{y_0})$, where $\alpha: (S_j, e) \to (F, \overline{y_0})$, is mapped to $[\beta] \in  \pi_{j+1}(N^k, N^k \setminus \Delta_k(N); y_0) $, where $\beta: (D^{j+1}, S^j, e) \to (N^k, N^k \setminus \Delta_k(N), y_0)$ is given by $$\beta(e+t(v-e))=\alpha(z)(1-t),$$ where $e=(1,0,\ldots , 0)$ and $v\in S^j$.

Now, since $(N^k, N^k \setminus \Delta_k(N))$ is $(n(k-1) -1)$-connected,  $F$ is $(n(k-1) -2)$-connected. Therefore, if $3 \le n(k-1)$ then $\pi_{(n(k-1) -1)}(F) $ forms a local system on $N^{k}$. Our next step is to describe such local system $\pi_{(n(k-1) -1)}(F) $ on $N^{k}$.

\section{The local system $\pi_{(n(k-1) -1)}(F)$}

As we saw in the last section, if we choose a base point $y_0=(y_0^1, \ldots , y_0^k) \in N^k \setminus \Delta_k(N)$, we have a natural identification $$\psi_{y_0}: \pi_{n(k-1)}(N^k, N^k \setminus \Delta_k(N); y_0)  \simeq  \pi_{n(k-1)-1}(F; \overline{y_0}),$$ where $F=p^{-1}(y_0)$ and $ \overline{y_0}$ is the constant path at $y_0$. In fact, $\psi$ is an isomorphism of local systems on $N^k\setminus \Delta_k(N)$. We now determine the action of $$\pi_1(N^k \setminus \Delta_k(N); y_0)=\pi_1(N; y_0^1) \times \cdots \times \pi_1(N; y_0^k) = \underbrace{ \pi \times \cdots \times \pi}_{k-times}$$ on $$ \pi_{n(k-1)} (N^k, N^k \setminus \Delta_k(N); y_0).$$

Let $\eta : \tilde{N} \to N$ be the universal cover of $N$. We assume that $y_0=(y_0^1, \ldots , y_0^k)$ lies inside a small tubular neighborhood of $\Delta_k(N)$ and we consider the diagram 

\begin{table}[h]
\centering
\begin{tabular}{c} \xymatrix{ \tilde{N} \ar[d]_{\eta} & (\tilde{N}^k, \tilde{N}^k \setminus \xi^{-1}(\Delta_k(N))) \ar[l] \ar[d]_\xi \\
N  & (N^k, N^k \setminus \Delta_k(N)) \ar[l] \\}
\end{tabular}
\end{table}

where $\xi= \underbrace{\eta \times \cdots \times \eta}_{k-times}$ and the horizontal maps are projections on the first coordinate. These horizontal maps are fibered pairs with fibers $(N^{k-1}, N^{k-1} \setminus \underbrace{(y_0^1, y_0^1, \ldots , y_0^1)}_{(k-1)-times})$ (lower horizontal map)  and $(\tilde{N}^{k-1}, \tilde{N}^{k-1} \setminus \underbrace{\eta^{-1}(y_0^1)\times \eta^{-1}(y_0^1) \times \cdots \times \eta^{-1}(y_0^1)}_{(k-1)-times} )$ (upper horizontal map). Then, choosing $\tilde{y}_0=(\tilde{y}_0^1, \ldots , \tilde{y}_0^k) \in \tilde{N}^k$, we have the isomorphisms
\begin{eqnarray*}
&& \pi_{n(k-1)}(\tilde{N}^{k-1}, \tilde{N}^{k-1} \setminus \eta^{-1}(y_0^1)\times \eta^{-1}(y_0^1) \times \cdots \times \eta^{-1}(y_0^1) ; (\tilde{y}_0^2, \ldots , \tilde{y}_0^k)) \\
& & \simeq  \pi_{n(k-1)}  (\tilde{N}^k, \tilde{N}^k \setminus \xi^{-1}(\Delta_k(N)); \tilde{y}_0)
\end{eqnarray*}
 and
\begin{eqnarray*}
& & \pi_{n(k-1)}(N^{k-1}, N^{k-1} \setminus (y_0^1, y_0^1, \ldots , y_0^1); (y_0^2, \ldots , y_0^k)) \\
& &\simeq  \pi_{n(k-1)} (N^k, N^k \setminus \Delta_k(N); y_0).
\end{eqnarray*}

Moreover, because $\xi: (\tilde{N}^k, \tilde{N}^k \setminus \xi^{-1}(\Delta_k(N))) \to  (N^k, N^k \setminus \Delta_k(N))$ is a covering map, we have 
\begin{eqnarray*}
 \pi_{n(k-1)} (\tilde{N}^k, \tilde{N}^k \setminus \xi^{-1}(\Delta_k(N)); \tilde{y}_0) \simeq  \pi_{n(k-1)} (N^k, N^k \setminus \Delta_k(N); y_0).
\end{eqnarray*}
Therefore 
\begin{eqnarray*}
& & \pi_{n(k-1)} (N^k, N^k \setminus \Delta_k(N); y_0) \\ & &  \simeq  \pi_{n(k-1)}(\tilde{N}^{k-1}, \tilde{N}^{k-1} \setminus \eta^{-1}(y_0^1)\times \eta^{-1}(y_0^1) \times \cdots \times \eta^{-1}(y_0^1) ; (\tilde{y}_0^2, \ldots , \tilde{y}_0^k)).
\end{eqnarray*}

Furthermore, by simple connectivity, it follows from the relative Hurewicz theorem that
\begin{eqnarray*}
& & \pi_{n(k-1)}(\tilde{N}^{k-1}, \tilde{N}^{k-1} \setminus \eta^{-1}(y_0^1)\times \eta^{-1}(y_0^1) \times \cdots \times \eta^{-1}(y_0^1) ; (\tilde{y}_0^2, \ldots , \tilde{y}_0^k))\\
& & \simeq H_{n(k-1)} (\tilde{N}^{k-1}, \tilde{N}^{k-1} \setminus \eta^{-1}(y_0^1)\times \eta^{-1}(y_0^1) \times \cdots \times \eta^{-1}(y_0^1) ).
\end{eqnarray*}

Note that $$  (\tilde{N}^{k-1}, \tilde{N}^{k-1} \setminus \eta^{-1}(y_0^1)\times \eta^{-1}(y_0^1) \times \cdots \times \eta^{-1}(y_0^1) ) = \underbrace{(\tilde{N}, \tilde{N} \setminus \eta^{-1}(y_0^1)) \times \cdots \times (\tilde{N}, \tilde{N} \setminus \eta^{-1}(y_0^1))}_{(k-1)-times}$$ and, therefore,
\begin{eqnarray*}
 && \pi_{n(k-1)} (N^k, N^k \setminus \Delta_k(N); y_0) \\ & & \simeq
\underbrace{ H_n(\tilde{N}, \tilde{N} \setminus \eta^{-1}(y_0^1) ) \otimes \cdots \otimes H_n(\tilde{N}, \tilde{N} \setminus \eta^{-1}(y_0^1) )}_{(k-1)-times}.
\end{eqnarray*}

Let $\eta^{-1}(y_0^1)=\{ \tilde{y}_i\}_i$ with the convention that $ \tilde{y}_1=\tilde{y}_0^1$. Let $V$ be an euclidean neighborhood of $y_0^1$ in $N$ and $\tilde{V}_i$ an euclidean neighborhood of $\tilde{y}_i$ such that $\eta^{-1}(V) = \sqcup_i  \tilde{V}_i$ and $\eta: \tilde{V}_i \to V$ is a homeomorphism. By excision, $$H_n(\tilde{N}, \tilde{N} \setminus \eta^{-1} (y_0^1)) \simeq \sum_i H_n (\tilde{V}_i, \tilde{V}_i \setminus \tilde{y}_i).$$

Identifying $\pi$ with the covering transformations of $\tilde{N}$, given $\sigma \in \pi$, if $\sigma \tilde{y}_0^1 = \tilde{y}_i$ then $\sigma \tilde{V}_1= \tilde{V}_i$.

Choose a local orientation of $V$ at  $y_0^1$, which determines a generator $\gamma_1\in H_n(\tilde{V}_1, \tilde{V}_1 \setminus \tilde{y}_0^1 )$. Let $\alpha \in \pi$ and set $\gamma_\alpha = \alpha \gamma_1$. Thus, $\gamma_\alpha$ generates $H_n(\tilde{V}_i, \tilde{V}_i \setminus \tilde{y}_i)$ if $\alpha \tilde{y}_0^1=\tilde{y}_i$.

Since $(\tilde{N}^{k-1}, \tilde{N}^{k-1} \setminus \eta^{-1}(y_0^1)\times \eta^{-1}(y_0^1) \times \cdots \times \eta^{-1}(y_0^1))$ is the fiber of the fiber pair map  $$\tilde{N} \leftarrow  (\tilde{N}^k, \tilde{N}^k \setminus \xi^{-1}(\Delta_k(N))) $$
over each $\tilde{y}_i$, for each  $\tilde{y}_i$ we have a fiber inclusion $$ \theta_i: (\tilde{N}^{k-1}, \tilde{N}^{k-1} \setminus \eta^{-1}(y_0^1)\times \eta^{-1}(y_0^1) \times \cdots \times \eta^{-1}(y_0^1)) \hookrightarrow (\tilde{N}^k, \tilde{N}^k \setminus \xi^{-1}(\Delta_k(N)))$$ given by $$\theta_i(u)=(\tilde{y}_i, u), \ u\in \tilde{N}^{k-1}.$$
Now, we identify $\Z [\underbrace{\pi \times \cdots \times \pi}_{(k-1)-times}]$ with the image of $$\underbrace{ H_n(\tilde{N}, \tilde{N} \setminus \eta^{-1}(y_0^1) ) \otimes \cdots \otimes H_n(\tilde{N}, \tilde{N} \setminus \eta^{-1}(y_0^1) )}_{(k-1)-times}$$ under ${\theta_1}_\ast$ via the correspondence $$ (\alpha_1, \ldots , \alpha_{k-1}) \mapsto {\theta_1}_\ast(\gamma_{\alpha_1} \otimes \cdots \otimes \gamma_{\alpha_{k-1}})= {\theta_1}_\ast (\alpha_1 \gamma_1 \otimes \cdots \otimes \alpha_{k-1} \gamma_1).$$ Observe that ${\theta_1}_\ast ( \gamma_1 \otimes \cdots \otimes  \gamma_1)$ can be represented by an $n(k-1)$-cell $\{\tilde{y}_0^1\}\times(D_1^n)^{k-1}$  in $\tilde{N}^k$ transverse to $\xi^{-1}(\Delta_k(N))$ at  $(\tilde{y}_0^1, \tilde{y}_0^1, \ldots , \tilde{y}_0^1)$. Thus, for $\sigma \in \pi$, the diagonal element $$(\sigma, \sigma, \ldots , \sigma) \in \underbrace{\pi \times \pi \times \cdots \times \pi}_{k-times}$$ sends $\{\tilde{y}_0^1\}\times(D_1^n)^{k-1}$ to an $n(k-1)$-cell $\{\tilde{y}_i\}\times(D_i^n)^{k-1}$ transverse to $\xi^{-1}(\Delta_k(N))$ at  $(\tilde{y}_i, \tilde{y}_i, \ldots , \tilde{y}_i)$ if $\sigma \tilde{y}_0^1=\tilde{y}_i $. Thus, one can see that $$ (\sigma, \sigma, \ldots , \sigma) {\theta_1}_{\ast} (\gamma_1 \otimes \gamma_1 \otimes \cdots \otimes \gamma_1)= (\text{sgn} \, \sigma )^{k-1} {\theta_1}_{\ast}  (\gamma_1 \otimes \gamma_1 \otimes \cdots \otimes \gamma_1).$$

Now, we can compute the action of $\pi^k$ on $\Z[\pi^{k-1}]$ via the above identifications:
\begin{equation}\label{local_system_action}
\begin{aligned}
(\sigma_1, \ldots , \sigma_k) \circ (\alpha_1, \ldots , \alpha_{k-1}) &\equiv (\sigma_1, \ldots , \sigma_k) {\theta_1}_\ast (\alpha_1 \gamma_1 \otimes \cdots \otimes \alpha_{k-1} \gamma_1)\\
&= (\sigma_1, \ldots , \sigma_k) (1\times \alpha_1 \times \cdots \times \alpha_{k-1}) {\theta_1}_\ast ( \gamma_1 \otimes \cdots \otimes  \gamma_1)\\
&= (\sigma_1, \sigma_2 \alpha_1, \ldots , \sigma_k \alpha_{k-1}) {\theta_1}_\ast ( \gamma_1 \otimes \cdots \otimes  \gamma_1)\\
&= (1, \sigma_2 \alpha_1 \sigma_1^{-1} , \ldots , \sigma_k \alpha_{k-1} \sigma_1^{-1}) (\sigma_1,  \ldots, \sigma_1)  {\theta_1}_\ast ( \gamma_1 \otimes \cdots \otimes  \gamma_1)\\
&= (\text{sgn} \, \sigma_1 )^{k-1} (1, \sigma_2 \alpha_1 \sigma_1^{-1} , \ldots , \sigma_k \alpha_{k-1} \sigma_1^{-1}) {\theta_1}_\ast ( \gamma_1 \otimes \cdots \otimes  \gamma_1)\\
&= (\text{sgn} \, \sigma_1 )^{k-1} {\theta_1}_\ast ( \sigma_2 \alpha_1 \sigma_1^{-1} \gamma_1 \otimes \cdots \otimes  \sigma_k \alpha_{k-1} \sigma_1^{-1} \gamma_1)\\
&\equiv (\text{sgn} \, \sigma_1 )^{k-1} ( \sigma_2 \alpha_1 \sigma_1^{-1}, \ldots , \sigma_k \alpha_{k-1} \sigma_1^{-1}).
\end{aligned}
\end{equation}

By pulling back the local system by the map $f:M\to N^k$, the action of $\pi_1(M)$ on $\mathbb Z[\pi^{k-1}]$ is given by
\begin{equation}\label{induced-local-system}
\begin{aligned}
\gamma \circ (\alpha_1, \ldots , \alpha_{k-1}) &=(\varphi_1(\gamma), ... ,\varphi_k(\gamma))\circ (\alpha_1, \ldots , \alpha_{k-1}) \\
                                          &=(\text{sgn} \, \varphi_1(\gamma) )^{k-1}(\varphi_2(\gamma) \alpha_1 \varphi_1(\gamma)^{-1},\ldots ,\varphi_{k}(\gamma) \alpha_{k-1} \varphi_1(\gamma)^{-1}).
\end{aligned}
\end{equation}

Here, $\varphi_i$ is the homomorphism induced by $f_i$.

\section{Lefschetz coincidence class as primary obstruction}

Following \cite{fadell-husseini}, the primary obstruction to deforming $f$ off a subspace $B$ was defined as a {\it universal element} in \cite{dobrenko}. In our setting, first we triangulate $N$ and hence $N^k$ so that the $(k-1)n$ skeleton $(N^k)^{((k-1)n)}\subset N^k\setminus \Delta_k(N)$ for $k\ge 3$. Then there is the (only) primary obstruction $\tilde \mu_k \in H^{(k-1)n}(N^k,N^k\setminus \Delta_k(N); \mathbb Z[\pi^{k-1}])$ to deforming the identity map $1_{N^k}$ off the subspace $\Delta_k(N)$. We call this element $\tilde \mu_k$ the {\it twisted Thom class} of the normal bundle of $\Delta_k(N)$ in $N^k$ or simply the {\it twisted Thom class} of $N^k$.

For any $i,j$, consider the map 
$$
e: (N^{i+j-1}, N^{i+j-1}\setminus \Delta_{i+j-1}(N))\to (N^{i}, N^{i}\setminus \Delta_{i}(N)) \times (N^{j}, N^{j}\setminus \Delta_{j}(N))
$$
given by $e(x_1,x_2,...,x_{i+j-1})=((x_1,x_2,...,x_i), (x_1,x_{i+1},x_{i+2},...,x_{i+j-1})$.

\begin{proposition}\label{product_Thom_class}
The map $e$ induces a homomorphism
$$H^{(i-1)n}(N^{\times_i}; \mathbb Z[\pi^{i-1}]) \otimes H^{(j-1)n}(N^{\times_j}; \mathbb Z[\pi^{j-1}]) \to H^{(i+j-2)n}(N^{\times_{i+j-1}}; \mathbb Z[\pi^{i+j-2}])
$$
such that $e^*(\tilde \mu_i \otimes \tilde \mu_j)=\tilde \mu_{i+j-1}$. Here $N^{\times_k}$ denotes the pair $(N^k, N^k\setminus \Delta_k(N))$ and $\pi=\pi_1(N)$.
\end{proposition}
\begin{proof}
First there is an Eilenberg-Zilber map 
$${\rm EZ}: C_*(N^{\times_i}; \mathbb Z[\pi^{i-1}]) \otimes C_*(N^{\times_j}; \mathbb Z[\pi^{j-1}]) \to C_*(N^{\times_{i+j}}; \mathbb Z[\pi^{i-1}]\otimes \mathbb Z[\pi^{j-1}])$$
where the action of $\pi^{i+j}=\pi^i\times \pi^j$ on $\mathbb Z[\pi^{i-1}]\otimes \mathbb Z[\pi^{j-1}]$ is the diagonal action. More precisely, for $(\sigma_1,...,\sigma_{i+j})\in \pi^{i+j}$, we have
\begin{equation}\label{ij-action}
\begin{aligned}
&(\sigma_1,..., \sigma_{i+j})\circ (\alpha_1,...,\alpha_{i-1},\beta_1,...,\beta_{j-1}) \\
                                         =&({\rm sgn} \sigma_1)^{i-1}(\sigma_2\alpha_1 \sigma_1^{-1},...,\sigma_i \alpha_{i-1}\sigma_1^{-1}) \otimes ({\rm sgn} \sigma_{i+1})^{j-1}(\sigma_{i+2}\beta_1 \sigma_{i+1}^{-1},...,\sigma_{i+j} \beta_{j-1}\sigma_{i+1}^{-1}).
\end{aligned}
\end{equation}
Now the map $e$ induces a local system $e^*(\mathbb Z[\pi^{i-1}]\otimes \mathbb Z[\pi^{j-1}])$ on $N^{\times_{i+j-1}}$ and this system is given by the following action:
\begin{equation*}
\begin{aligned}
&(\sigma_1,...,\sigma_{i+j-1})\ast (\alpha_1,...,\alpha_{i-1},\beta_1,...,\beta_{j-1}) \\
=&e_{\#}(\sigma_1,...,\sigma_{i+j-1})\circ (\alpha_1,...,\alpha_{i-1},\beta_1,...,\beta_{j-1}) \\
=&(\sigma_1,...,\sigma_i, \sigma_1, \sigma_{i+1},...,\sigma_{i+j-1})\circ (\alpha_1,...,\alpha_{i-1},\beta_1,...,\beta_{j-1}) \\
=&({\rm sgn} \sigma_1)^{i-1}(\sigma_2\alpha_1 \sigma_1^{-1},...,\sigma_i \alpha_{i-1}\sigma_1^{-1}) \cdot ({\rm sgn} \sigma_{1})^{j-1}(\sigma_{i+1}\beta_1 \sigma_{1}^{-1},...,\sigma_{i+j-1} \beta_{j-1}\sigma_{1}^{-1})
\end{aligned}
\end{equation*}
This action coincides with that of \eqref{local_system_action} so that $e^*(\mathbb Z[\pi^{i-1}]\otimes \mathbb Z[\pi^{j-1}])$ coincides with the local system $\mathbb Z[\pi^{i+j-2}]$ discussed above. Now the obstruction to deforming the identity map $1_{N^q}$ off the subspace $\Delta_q(N)$ has a simple cochain representation given by $c_{(q-1)n}(1_{N^q})(\sigma)=[\sigma] \in \pi_{(q-1)n}(N^{\times_q})$. Let $\sigma_1 \times \sigma_2 \times ...\times \sigma_{i+j-1}$ be an $(i+j-1)n$ simplex where each $\sigma_i$ is an $n$-simplex in $N$ and $\tilde c_p$ denote the cochain representing the twisted Thom class $\tilde \mu_p$. Then 
\begin{equation*}
\begin{aligned}
\langle \tilde c_{i+j-1}, \sigma_1 \times \sigma_2 \times ...\times \sigma_{i+j-1}\rangle &= 
[\sigma_1 \times \sigma_2 \times ... \times \sigma_i\times \sigma_{i+1} \times ... \times \sigma_{i+j-1}] \qquad \in \mathbb Z[\pi^{i+j-2}] \\
&= e^*([\sigma_1 \times \sigma_2 \times ... \times \sigma_i] \otimes [\sigma_1 \times \sigma_{i+1} \times ... \times \sigma_{i+j-1}] ) \\
&= e^*(\langle \tilde c_{i}, \sigma_1 \times \sigma_2 \times ... \times \sigma_i\rangle \otimes \langle \tilde c_j, \sigma_1 \times \sigma_{i+1} \times ... \times \sigma_{i+j-1}]\rangle) \\
&= e^*(\langle \tilde c_{i}\otimes \tilde c_{j}, (\sigma_1 \times \sigma_2 \times ... \times \sigma_i)\otimes (\sigma_1 \times \sigma_{i+1} \times ... \times \sigma_{i+j-1})\rangle) \\
&= \langle \tilde c_{i}\otimes \tilde c_{j}, e_*((\sigma_1 \times \sigma_2 \times ... \times \sigma_i)\otimes (\sigma_1 \times \sigma_{i+1} \times ... \times \sigma_{i+j-1}))\rangle.
\end{aligned}
\end{equation*}
It follows that $e^*(\tilde \mu_i \otimes \tilde \mu_j)=\tilde \mu_{i+j-1}$.
\end{proof}

As an immediate corollary, we have the following useful result.

\begin{corollary} \label{product_Thom_classes}
Let $e': N^{\times_k} \to \underbrace{N^{\times_2} \times ... \times N^{\times_2}}_{k-1}$ be defined by
$$
e'(x_1,...,x_k)=((x_1,x_2),(x_1,x_3),...,(x_1,x_k)).
$$
Suppose $n=\dim N\ge 3$. Then $e'$ induces a homomorphism $e'^*$ such that $e'^*(\underbrace{\tilde \mu_2 \otimes ... \otimes \tilde \mu_2}_{k-1})=\tilde \mu_{k}$.
\end{corollary}

Using the twisted Thom class as an element in the cohomology of $N^{\times_k}$ with local coefficients $\mathbb Z[\pi^{k-1}]$, we define the {\it twisted Lefschetz coincidence class of $f_1,...,f_k$} to be the element $\mathcal L(f_1,...,f_k):=\Lambda (f_1,...,f_k; \mathbb Z[\pi^{k-1}])(\tilde \mu_k)$. Since $\tilde \mu_k$ is the obstruction to deforming the identity off the subspace $\Delta_k(N)$, it follows from \cite{dobrenko} that $\mathcal L(f_1,...,f_k)=o_{(k-1)n}(f_1,...,f_k)$ the primary obstruction to deforming $f_1,...,f_k$ to be coincidence free on the $(k-1)n$ skeleton of $X$.

In the case when $N$ is oriented, it was already shown in \cite{GJW} that the primary obstruction to deforming $f$ and $g$ is mapped to the classical Lefschetz coincidence number $L(f,g)$ under the augmentation homomorphism $\mathbb Z[\pi] \to \mathbb Z$. Next, we show an analogous result, that is, there is a natural homomorphism $\tilde \eta$ induced by the augmentation map such that $\tilde \eta (\mathcal L(f_1,...,f_k))=L_2(f_1,...,f_k)$. Here, we assume the principal ideal domain $R$ is $\mathbb Z$.

\begin{theorem}\label{augmentation}
The augmentation map $\epsilon^{k-1}: \mathbb Z[\pi^{k-1}] \to \Gamma^*_N \times ... \times \Gamma^*_N$ induces a homomorphism $\tilde \eta$ such that $\tilde \eta(\mathcal L(f_1,...,f_k))=L_2(f_1,...,f_k)$.
\end{theorem}
\begin{proof}
By the Alexander duality isomorphism (see \cite[Thm. 6.4]{Spanier}), we have the following commutative diagram.
\begin{equation*}
\begin{CD}
    H^{(k-1)n}(N^{\times_{k}};\mathbb Z[\pi^{k-1}])    @>{A_1^{-1}}>>   H_n(\Delta_k(N); \Gamma_{N^k}\otimes \mathbb Z[\pi^{k-1}]) \\
    @V{\eta}VV  @VV{1\otimes \epsilon^{k-1}}V  \\
    H^{(k-1)n}(N^{\times_{k}};\mathbb Z\times \Gamma^*_N \times ... \times \Gamma^*_N)    @<{A_2}<<     H_n(\Delta_k(N); \Gamma_{N^k}\otimes \mathbb Z\times \Gamma^*_N \times ... \times \Gamma^*_N)
    \end{CD}
\end{equation*}
Here, $A_1$ and $A_2$ are the corresponding duality isomorphisms so that $\eta=A_2 \circ (1\otimes \epsilon^{k-1}) \circ A_1^{-1}$ where $\epsilon^{k-1}$ is the augmentation map. For any $\sigma \in \pi$, the action of $(\sigma, ..., \sigma)\in \pi_1(\Delta_k(N))$ on $\mathbb Z\times \Gamma^*_N \times ... \times \Gamma^*_N$ is the same as the action on $\mathbb Z[\pi^{k-1}]$ followed by $\epsilon^{k-1}$. This implies that the twisted fundamental class of $\Delta_k(N)$ with coefficients in $\Gamma_{N^k}\otimes \mathbb Z[\pi^{k-1}]$ is mapped under $1\otimes \epsilon^{k-1}$ to the twisted fundamental class of $\Delta_k(N)$ with coefficients in $\Gamma_{N^k}\otimes \mathbb Z\times \Gamma^*_N \times ... \times \Gamma^*_N$. By duality, the twisted Thom class $\tilde \mu_k$ is mapped under $\eta$ to $\mu_k$. Define $\tilde \eta$ by $\langle \tilde \eta (m),\sigma\rangle:= \eta (\langle m, \sigma\rangle )$. Pulling back the classes $\tilde \mu_k, \mu_k$ yields the assertion.\end{proof}

\begin{theorem}\label{obstruction-product} Let $f_1,..., f_k:X \to N$ be maps from a finite complex $X$ to a closed connected manifold $N$ of dimension $n$, $n\ge 3$. Then 
$$
o_{(k-1)n}(f_1,...,f_k)=o_n(f_1,f_2) \cup o_n(f_1,f_3) \cup ... \cup o_n(f_1,f_k).
$$
\end{theorem}
\begin{proof} In Corollary \ref{product_Thom_classes} we state that if $e': N^{\times_k} \to \underbrace{N^{\times_2} \times ... \times N^{\times_2}}_{k-1}$ is the map defined by $e'(x_1,...,x_k)=((x_1,x_2),(x_1,x_3),...,(x_1,x_k))$
then $e'^*(\underbrace{\tilde \mu_2 \otimes ... \otimes \tilde \mu_2}_{k-1})=\tilde \mu_{k}$.  Let us consider the inclusions $I: X \to (X, X\setminus C(f_1, \ldots, f_k) )$ and $j_i: X \to (X, X \setminus C(f_1, f_i))$, $i=2, \ldots , k$. Then, $$o_{(k-1)n}(f_1,...,f_k) = I^\ast (f_1, \ldots , f_k)^\ast (\tilde \mu_k) $$ and $$o_n(f_1,f_i)= j_i^\ast (f_1, f_i)^\ast (\tilde \mu_2), \ i=2, \ldots , k.$$ Note that $e' \circ (f_1, \ldots , f_k) \circ I = ((f_1, f_2) \times \ldots \times (f_1, f_k)) \circ (j_1, \ldots , j_{k-1})$. Therefore
\begin{eqnarray*}
o_{(k-1)n}(f_1,...,f_k) &=& I^\ast (f_1, \ldots , f_k)^\ast (\tilde \mu_k) \\
&=& I^\ast (f_1, \ldots , f_k)^\ast (e'^*(\tilde \mu_2 \otimes ... \otimes \tilde \mu_2))\\
&=& (j_1, \ldots , j_{k-1})^\ast  \circ ((f_1, f_2) \times \ldots \times (f_1, f_k))^\ast (\tilde \mu_2 \otimes ... \otimes \tilde \mu_2)\\
&=& j_1^\ast (f_1, f_2)^\ast (\tilde \mu_2) \cup j_2^\ast (f_1, f_3)^\ast (\tilde \mu_2) \cup ... \cup j_{k-1}^\ast (f_1, f_k)^\ast (\tilde \mu_2)\\
&=& o_n(f_1,f_2) \cup o_n(f_1,f_3) \cup ... \cup o_n(f_1,f_k)
\end{eqnarray*}
\end{proof}

\begin{remark} Since $o_{(k-1)n}(f_1,...,f_k)=\mathcal L(f_1,..., f_k)$ and $o_n(f_1,f_i)=\mathcal L(f_1,f_i)$, Theorem \ref{obstruction-product} states that for $n\ge 3$, we have
\begin{equation}\label{twisted_L_product}
\mathcal L(f_1,...,f_k)=\mathcal L(f_1,f_2) \cup \mathcal L(f_1,f_3) \cup ... \cup \mathcal L(f_1,f_k).
\end{equation}

Combining with the homomorphism $\tilde \eta$ as in Theorem \ref{augmentation}, it is easy to establish the following product
$$
L_2(f_1,...,f_k)=L_2(f_1,f_2) \cup L_2(f_1,f_3) \cup ... \cup L_2(f_1,f_k)
$$
which is obtained in \cite[Thm. 4.2]{MS}.
\end{remark}

\section{Converse of Lefschetz Coincidence Theorem}
The converse of the Lefschetz coincidence theorem in our setting amounts to studying the problem when the maps $f_1,...,f_k$ can be deformed to be coincidence free. For instance in the classical case, it has been shown \cite{fadell} that when $N$ is simply-connected then the vanishing of the Lefschetz coincidence number provides a converse to the Lefschetz coincidence theorem. To study this problem for multiple maps, we use the obstruction theory developed in the previous sections. 

Next, we need to calculate the obstruction cocycle, following \cite[pp.20-21]{dobrenko}.
Suppose $M$ is a closed $(k-1)n$-dimensional manifold. By general position, we may assume that the map $f=(f_1,...,f_k):M\to N^k$ is transverse to the submanifold $\Delta_k(N)$ so that $C=f^{-1}(\Delta_k(N))=C(f_1,...,f_k)$ is a finite set of coincidence points of $f_1,\ldots,f_k$. Let $C=\{s_1,...,s_r\}$. For each isolated coincidence point $s_i$, we may assume that there is a maximal $(k-1)n$-simplex $\sigma_i$ containing $s_i$ such that $C\cap \overline \sigma_i=\{s_i\}$. Following \cite{fadell-husseini,dobrenko}, the primary obstruction cocycle $c_{(k-1)n}(f)$ to deforming $f$ into the subspace $N^k\setminus \Delta_k(N)$ is given by
\begin{equation}\label{local-cocycle}
c_{(k-1)n}(f)(\sigma)=\left\{ 
\aligned
& 0 \qquad & \text{if $\sigma\cap C=\emptyset$} \\
& \lambda (f,s_i)[f|_{\sigma}] \qquad & \text{if $\sigma=\sigma_i$.}
\endaligned
\right.
\end{equation}
     
Here, $f|_{\sigma_i}$ represents the restriction map from $(\sigma_i,\partial \sigma_i)$ to $(V_i\times ... \times V_i, V_i\times ... \times V_i \setminus \{(y_i,..., y_i)\})$ where $V_i$ is an open neighborhood of $y_i=f_1(s_i)= ... =f_k(s_i)$ in $N$. Thus, $[f|_{\sigma_i}]\in \pi_{(k-1)n}(N^k,N^k\setminus \Delta_k(N))\cong \pi_{(k-1)n-1}(F)$. Furthermore the coefficient $\lambda (f,s_i)=f|_{\sigma_i}^*(\tilde \mu_k)$ is the {\it local coincidence index} defined by the restriction $f|_{\sigma_i}$ using the homomorphism 
$$
f|_{\sigma_i}^*: H^{(k-1)n}(N^k,N^k\setminus \Delta_k(N)) \to H^{(k-1)n}(\sigma_i,\sigma_i\setminus \{s_i\}).
$$
Since $(M,M-C)=(M,M-\{s_1,...,s_r\})$, it follows from excision and additivity that 
$$
I^*\circ f^* (\tilde \mu_k)=\sum_{i=1}^r \lambda (f,s_i)
$$
where $I:M\hookrightarrow (M,M-C)$.
It follows that 
\begin{equation}\label{cocycle}
c_{(k-1)n}(f)=\sum_i^r \lambda(f,s_i) [f|_{\sigma_i}] 
\end{equation}
and
\begin{equation}\label{Lef-sum}
\mathcal L(f_1,...,f_k)=\sum_{i=1}^r \lambda (f,s_i).
\end{equation}

\begin{theorem}\label{converse1}
Let $f_1,\ldots,f_k:M\to N$ from a closed connected $(k-1)n$-manifold $M$ to a closed connected $n$-manifold $N$. Then $f_1,\ldots,f_k$ are deformable to be coincidence free if, and only if, $o_{(k-1)n}(f_1,...,f_k)=0$ where the obstruction cocycle is given by \eqref{cocycle}.
\end{theorem}

In the case where $N$ is simply-connected, $o_{(k-1)n}(f_1,...,f_k)=\mathcal L(f_1,...,f_k)=L_2(f_1,...,f_k)=L_1(f_1,...,f_k)$ since $N$ is orientable and the coefficients are simple. In particular, when $N=S^n$ is the $n$-sphere, the formula of \cite[Prop. 3.2]{BLM} becomes
$$
o_{(k-1)n}(f_1,...,f_n)=\sum_{i=0}^{k-1} (-1)^{(i-1)n}\widehat{f_{k-i}}^*(e)
$$
where $e=\underbrace{ e_n \times ... \times e_n}_{(k-1)-times}$, $e_n$ the cohomology fundamental class of $S^n$, $\widehat{f_{j}}=(f_1,..., f_{j-1},f_{j+1},..., f_k)$.
This formula generalizes a similar formula in \cite[Theorem 3.4]{GJW}.
Furthermore, the equality $o_{(k-1)n}(f_1,...,f_k)=L_2(f_1,\ldots,f_k)$ generalizes \cite[Theorem 1.1]{fadell}) and we have the following converse theorem of the Lefschetz coincidence theorem when $N$ is simply connected.

\begin{theorem}\label{converse2}
Let $f_1,\ldots,f_k:M\to N$ from a closed connected $(k-1)n$-manifold $M$ to a closed connected and simply-connected $n$-manifold $N$. Then $f_1,\ldots,f_k$ are deformable to be coincidence free if, and only if, $L_2(f_1,...,f_k)=0$.
\end{theorem}

To obtain the converse Lefschetz theorem, Nielsen coincidence theory is often employed. In \cite{St}, P. Staecker independently investigated the coincidence problem for multiple maps from the Nielsen coincidence theory point of view. More precisely, Staecker considered the following two maps $F=(f_1,...,f_1), G=(f_2,...,f_k):M\to N^{k-1}$ and employed the usual Nielsen coincidence theory by defining the Lefschetz and Nielsen coincidence numbers $L(F,G)$ and $N(F,G)$. When $N$ is non-orientable, an appropriate (semi)index was used to define essentiality of the coincidence classes. We should point out that when $N$ is orientable, $|L_2(f_1,...,f_k)|=|L(F,G)|$ and $o_{(k-1)n}(f_1,...,f_k)=0 \Leftrightarrow N(F,G)=0$. In this case, since $C=C(f_1,...,f_k)=\{s_1,...,s_r\}$ is a finite set, $C$ is a disjoint union of $N(F,G)$ coincidence classes. Thus, the coincidence classes will have index of the same sign, if $N$ is a Jiang space; a nilmanifold \cite{daci-peter3}, an orientable coset space $G/K$ of a compact connected Lie group $G$ by a closed subgroup $K$ \cite{VW}; or a $\mathcal C$-nilpotent space whose fundamental group has a finite index center \cite{daci-peter4} where $\mathcal C$ denotes the class of finite groups. For these spaces, $L(F,G)=0 \Rightarrow N(F,G)=0$. Thus, we have the following converse theorem.

\begin{theorem}\label{converse3}
Let $f_1,\ldots,f_k:M\to N$ from a closed connected $(k-1)n$-manifold $M$ to a closed connected orientable $n$-manifold $N$. Suppose $N$ is a Jiang space; a nilmanifold, an orientable coset space $G/K$ of a compact connected Lie group $G$ by a closed subgroup $K$; or a $\mathcal C$-nilpotent space whose fundamental group has a finite index center where $\mathcal C$ is the class of finite groups. Then $f_1,\ldots,f_k$ are deformable to be coincidence free if, and only if, $L_2(f_1,...,f_k)=0$.
\end{theorem}

\section{Examples}

When $n=2$ and $k\ge 3$, $(k-1)n\ge 3$ so that $\pi_{(k-1)n-1}(F)\cong \pi_{(k-1)n}(N^k, N^k\setminus \Delta_k(N))$ is an abelian group so it is a local coefficient system on $N^k$ even when $N$ is a surface.  Thus, Theorem \ref{converse2} and Theorem \ref{converse3} are valid when $n=2$ provided $k\ge 3$. Next we illustrate by an example in which $f_1,f_2,f_3: X \to S_g$ are maps from a closed 4-manifold to a surface $S_g$ such that $f_1,f_i$ cannot be deformed to be coincidence free for $i=2,3$. However, $f_1,f_2,f_3$ are deformable to be coincidence free.

\begin{example}\label{2-sphere}
Let $X=S^1\times S^3$, $f_1:X\to S^2$ be given by  $f_1(x_1,x_2)=h(x_2)$ where $x_2\in S^3$ and $h:S^3\to S^2$ is the Hopf map. Let $f_2=\overline c$ be the constant map at some $c\in S^2$. Note that $f_1:X\to S^2$ is a fibration and $f_1^*([\mu])=0$ where $[\mu]$ is the cohomology fundamental group of the base space $S^2$. This follows from the fact that $H^2(X)=0$. Thus, we conclude that $L_2(f_1,f_2)=0$. Note that $f_1$ and $f_2$ cannot be deformed to be coincidence free.  
Now let $f_3:X\to S^2$ be {\it any} map. Since $L_2(f_1,f_2)=0$, by \cite[Thm. 4.2]{MS}, we have $L_2(f_1,f_2,f_3)=0$. The 2-sphere $S^2$ is simply connected so by Theorem \ref{converse2}, $f_1,f_2,f_3$ are deformable to be coincidence free. In particular, if we choose $f_3=f_2$ then $f_1,f_i$ are not deformable to be coincidence free for $i=2,3$. In fact, any maps $f_1,f_2,f_3: X\to S^2$ are deformable to be coincidence free. Note that $X$ is not symplectic so by the classical condition of Thurston \cite{T}, the fiber $S^1\times S^1$ is null homologous in $X$. As it turns out, Thurston's condition is only necessary for $X$ to be symplectic as we use that to contruct the next example.
\end{example}

\begin{example}\label{torus}
According to R. Geiges \cite{G}, there exist symplectic $4$-manifolds $X$ that is a principal $T^2$-bundle over $T^2$ but the fibration is not symplectic so that $i_*([F])=0$ in $H_2(X)$ where $[F]$ is the fundamental class of the fiber $F=T^2$. According to a result of Gottlieb (see its generalization in \cite{daci-peter}), if $p:X\to T^2$ is such a bundle then $p^*([T^2])=0$. It follows that $L_2(p,\overline c)=0$ for any $c\in T^2$ (base). On the other hand, $T^2$ is aspherical, so it follows from the main result of \cite{daci-peter5} that the (abelianized) obstruction $\overline o_2(p,\overline c)\ne 0$. Again, for {\it any} $f:X\to T^2, L_2(p,\overline c,f)=0$. Since $T^2$ is a Jiang space, it follows from Theorem \ref{converse3} that $p,\overline c, f$ are deformable to be coincidence free.
\end{example} 

We can now generalize the examples above as follows.

\begin{example}\label{general_example}
Let $M=M^{(k-1)n}$ be an even dimensional closed connected oriented manifold and $p:M \to N$ is a fibration over a closed connected oriented $n$-manifold $N$ with a typical fiber $F$, a closed connected oriented manifold. Asssume that $M$ is not symplectic, $N$ is aspherical, and is a Jiang-type space as in Theorem \ref{converse3}. For any $c\in N$ and any maps $f_3,..., f_k:M\to N$, $L_2(p,\overline c, f_3,...,f_k)=0$ so that $p,\overline c, f_3,..., f_k$ are deformable to be coincidence free but $p, \overline c$ cannot be homotopic to coincidence free maps.
\end{example}

\begin{remark}
Following an observation of F.B. Fuller, Brooks \cite{B1} showed that if $C(f',g')$ denotes the coincidence set for $f'\sim f, g'\sim g$ then there exists $g''\sim g$ such that $C(f',g')=C(f,g'')$. In other words, deforming both maps can be achieved by deforming only one of the two maps. In this multiple map setting, we can ask the same question:
Suppose $f_1\sim f_1',..., f_k\sim f_k'$ with $C(f_1',..., f_k')$. Can we obtain the same coincidence set by fixing more than one of the $k$-maps?
Using the last example with $k=3$, we take the three maps to be either (1) $p,p,c$ or (2) $c,c, p$ where $c$ denotes the constant map. If we fix the first two maps in each case, deforming the third map will {\it always} yield coincidences. This means that one can fix at most {\it one} map without changing the coincidence set so that the main result of \cite{B1} cannot be improved. Furthermore, since the approach of Staecker \cite{St} is to consider the (codimension zero) coincidence problem for the maps $F=(f_1,...,f_1)$ and $G=(f_2,...,f_k)$, our discussion above shows that if $f_1,...,f_k$ are deformable to be coincidence free then $G$ cannot be kept fixed, in other words, one {\it must} deform $G$ to some $G'$ with $C(F,G')=\emptyset$.
\end{remark}

\end{document}